\documentclass[fleqn, reqno]{amsart}

\usepackage{amsmath,amsthm,amsfonts,latexsym,euscript,amssymb,amscd}

\newtheorem{theorem}{Theorem}[section]
\newtheorem{corollary}[theorem]{Corollary}
\newtheorem{lemma}[theorem]{Lemma}
\newtheorem{proposition}[theorem]{Proposition}

\theoremstyle{definition}
\newtheorem{notation}[theorem]{Notation}
\newtheorem{definition}[theorem]{Definition}
\newtheorem{remark}[theorem]{Remark}
\newtheorem{remarks}[theorem]{Remarks}
\newtheorem{example}[theorem]{Example}
\newtheorem{examples}[theorem]{Examples}
\newtheorem{counter example}[theorem]{Counter Example}
\newtheorem{question}[theorem]{Question}

\newtheorem{problem}[theorem]{Problem}
\newtheorem{problems}[theorem]{Problems}
\newtheorem{conjecture}[theorem]{Conjecture}

\newcommand{\bgth}{\begin{theorem}}
\newcommand{\ndth}{\end{theorem}}
\newcommand{\bgthm}{\begin{theorem}}
\newcommand{\ndthm}{\end{theorem}}
\newcommand{\bgcor}{\begin{corollary}}
\newcommand{\ndcor}{\end{corollary}}
\newcommand{\bglm}{\begin{lemma}}
\newcommand{\ndlm}{\end{lemma}}
\newcommand{\bgprop}{\begin{proposition}}
\newcommand{\ndprop}{\end{proposition}}

\newcommand{\bgdf}{\begin{definition}}
\newcommand{\nddf}{\end{definition}}
\newcommand{\bgnota}{\begin{notation}}
\newcommand{\ndnota}{\end{notation}}
\newcommand{\bgrmk}{\begin{remark}}
\newcommand{\ndrmk}{\end{remark}}
\newcommand{\bgrmks}{\begin{remarks}}
\newcommand{\ndrmks}{\end{remarks}}
\newcommand{\bgexm}{\begin{example}}
\newcommand{\ndexm}{\end{example}}
\newcommand{\bgexms}{\begin{examples}}
\newcommand{\ndexms}{\end{examples}}

\newcommand{\bgques}{\begin{question}}
\newcommand{\ndques}{\end{question}}
\newcommand{\bgquess}{\begin{question}}
\newcommand{\ndquess}{\end{question}}
\newcommand{\bgprob}{\begin{problem}}
\newcommand{\ndprob}{\end{problem}}
\newcommand{\bgprobs}{\begin{problems}}
\newcommand{\ndprobs}{\end{problems}}
\newcommand{\bgconj}{\begin{conjecture}}
\newcommand{\ndconj}{\end{conjecture}}

\newcommand{\bgeq}{\begin{eqnarray}}
\newcommand{\ndeq}{\end{eqnarray}}
\newcommand{\bgeqq}{\begin{eqnarray*}}
\newcommand{\ndeqq}{\end{eqnarray*}}
\def\CC{{\mathbb C}}
\def\RR{{\mathbb R}}

\def\tr{{\rm tr}}

\numberwithin{equation}{section}

\newcommand{\lmref}[1]{Lemma~\ref{#1}}
\newcommand{\propref}[1]{Proposition~\ref{#1}}

\newcommand{\thmref}[1]{Theorem~\ref{#1}}

\title[Factorization Property for Quantum Groups]
{Kirchberg's Factorization Property for \\ 
Discrete Quantum Groups}
\author[Angshuman Bhattacharya and Shuzhou Wang]
{Angshuman Bhattacharya and Shuzhou Wang}

\address{Department of Mathematics, University of Georgia,
Athens, GA 30602
\newline \indent
Fax: 706-542-2573; Tel: 706-542-0884
}
\email{angshu@uga.edu, szwang@uga.edu}



\begin{document}

\begin{abstract}
We show that the discrete duals of the universal unitary quantum groups $U_n^+$ and orthogonal quantum groups
$O_n^+$ have Kirchberg's factorization property when $n \neq 3$. 

\end{abstract}

\maketitle

\section{Introduction}
In \cite{Tak64}, 
using left $\lambda$ and right $\rho$ regular  representations of the free group on two generators ${\mathbb F}_2$, 
Takesaki discovered that the product representation of the algebraic tensor product 
of the {\em reduced} algebras $\lambda(C^*({\mathbb F}_2)) \otimes_{\rm alg} \rho(C^*({\mathbb F}_2))$   
in $B(\ell^2({\mathbb F}_2))$, defined on simple tensors by 
$$
\lambda(C^*({\mathbb F}_2)) \otimes_{\rm alg} \rho(C^*({\mathbb F}_2))  \ni 
\lambda(a) \otimes \rho(b) \longmapsto \lambda(a)  \rho(b) \in  B(\ell^2({\mathbb F}_2)), 
\; \; 
$$
is {\em unbounded} for the minimun $C^*$-tensor norm, 
thus giving rise to a $C^*$-norm on  this tensor product 
that is different from the minimum norm. More than ten years later, 
replacing the reduced group $C^*$-algebra of ${\mathbb F}_2$ by the full group $C^*$-algebra, 
 Wassermann \cite{SWass76} showed that the product representation $(\lambda \cdot \rho)_{\rm alg}$ of 
the algebraic tensor product  
 $C^*({\mathbb F}_2) \otimes_{\rm alg} C^*({\mathbb F}_2)$, defined on simple tensors by,  
 $$
(\lambda \cdot \rho)_{\rm alg}: 
C^*({\mathbb F}_2) \otimes_{\rm alg} C^*({\mathbb F}_2) \ni a \otimes b  
\longmapsto \lambda(a)  \rho(b) \in  B(\ell^2({\mathbb F}_2)),  
 \; \;  
 $$ 
 is {\em bounded} for the corresponding 
 minimun $C^*$-tensor norm; that is, the representation $(\lambda \cdot \rho)_{\rm alg}$ 
 has a ({\em bounded}) extension  to the minimal $C^*$-tensor product 
$$
 (\lambda \cdot \rho)_{\rm min}:  C^*({\mathbb F}_2) \otimes_{\rm min} C^*({\mathbb F}_2) \rightarrow B(\ell^2({\mathbb F}_2)). 
$$ 
 
Motivated by these results, Kirchberg \cite{Kirch94} 
defined that a discrete group $\Gamma$ has {\em the factorization property} if 
the representation $(\lambda \cdot \rho)_{\rm alg}$ of 
$C^*(\Gamma) \otimes_{\rm alg} C^*(\Gamma)$, defined as above with ${\mathbb F}_2$ replaced by 
$\Gamma$, has a  ({\em bounded}) extension to the minimal $C^*$-tensor product
%$$
\begin{eqnarray} \label{factrep}
(\lambda \cdot \rho)_{\rm min}:  C^*(\Gamma) \otimes_{\rm min} C^*(\Gamma) \rightarrow B(\ell^2(\Gamma))
\end{eqnarray}
%$$
so that the representation of the maximal $C^*$-tensor product
%$$
\begin{eqnarray} \label{maxrep}
(\lambda \cdot \rho)_{\rm max}: C^*(\Gamma) \otimes_{\rm max} C^*(\Gamma) \rightarrow B(\ell^2(\Gamma))
\end{eqnarray}
%$$
admits a {\em factorization} of the form 
$(\lambda \cdot \rho)_{\rm max} = (\lambda \cdot \rho)_{\rm min} \circ \pi$, where 
\begin{eqnarray}
\pi: C^*(\Gamma) \otimes_{\rm max} C^*(\Gamma) \rightarrow C^*(\Gamma) \otimes_{\rm min} C^*(\Gamma)
\end{eqnarray}
is the canonical surjection. 
\iffalse
Observe that the existence of extension
$\lambda \otimes_{\rm min} \rho$ is all we need for $\Gamma$ to have the 
factorization property since $\lambda \otimes_{\rm max} \rho = (\lambda \otimes_{\rm min} \rho) \circ P$ 
follows automatically. 
\fi
 Kirchberg showed in \cite{Kirch94} that $\Gamma$ has 
the factorization property if and only if the Plancherel trace 
on $C^*(\Gamma)$ is amenable, extending Connes 
\cite{ConInj} (see Section 2 below for details on amenable traces), and 
that a group with Kazhdan's property T has the factorization property if and only if it is residually finite. 
So far little is known about this mysterious class of groups beyond Kirchberg's fundamental work, except 
the counter examples in \cite{Thom14}.

In \cite{Wor87b}, 
%using the quantum analog of continuous functions on a compact group, 
Woronowicz introduced the abstract theory of compact quantum groups that includes 
compact groups, Kac type compact quantum groups 
\cite{EnSch}, and the well known deformations of compact Lie groups \cite{Wor87a,Wor88a}.  
Borrowing the notation from compact abelian groups, 
for a compact quantum group $G$, 
we use the notation $C(G) = C^*(\hat{G})$, where $\hat{G}$ is the discrete quantum group dual to $G$. 
It is known that all discrete quantum groups can be obtained in this fashion. 
Assume $G$ is a Kac type compact quantum group. 
 % (we also say $G$ is a compact quantum group of Kac type),  
Then the antipode $\kappa$ on $C(G)$ is bounded and involutive and its Haar state  $h$ is a trace. 
Denote $\hat{G}$ by $\Gamma$ and let $\lambda$ be the GNS representation associated with 
$h$, $\lambda(a) \Lambda_h(b) = \Lambda_h(ab)$, where $a, b \in C^*(\Gamma)= C(G) $, 
$\Lambda_h(b)$ is the image of $b$ in the GNS Hilbert space $\ell^2(\Gamma) := L^2(C(G), h)$.  
%which is also denoted by $\ell^2(\Gamma)$. 
Define $\rho(a) \Lambda_h(b)  = \Lambda_h(b \kappa(a))$. 
As in the case of a discrete group, 
we define that a discrete quantum group $\Gamma$ has {\em Kirchberg's factorization property} if 
 the representation $(\lambda \cdot \rho)_{\rm alg}$ of 
 $C^*(\Gamma) \otimes_{\rm alg} C^*(\Gamma) = C(G)  \otimes_{\rm alg} C(G) $ has a (bounded) extension of the form (\ref{factrep}).

Woronowicz's abstract theory led to the discovery of 
a different kind of compact quantum groups, $A_u(n)$ and $A_o(n)$ (cf. \cite{free}), called 
the universal unitary quantum groups and universal orthogonal quantum groups.  
 In \cite{Banica96a, Banica97a}, Banica conducted a deep study of these quantum groups that 
 laid foundation for much of the subsequent work on these objects. 
Banica and his collaborators also invented the notation $U_n^+$ and $O_n^+$ for these 
 quantum groups to denote $A_u(n) = C(U_n^+)$ and $A_o(n) = C(O_n^+)$.
 %, which we will follow here. 
 In this paper, we show that the discrete quantum group duals $\widehat{U_n^+}$ and $\widehat{O_n^+}$ 
 has Kirchberg's factorization property when $n \neq 3$.
 
 This paper is organized as follows. In Section 2, we collect some basic concepts and 
 results for convenience of the reader. {\em All compact quantum groups are assumed to be 
 of Kac type in this paper.} In Section 3, we prove Theorem 3.3 which states 
 that the Haar trace of a compact quantum group is amenable if the quantum group 
 is generated in the sense of Brannan-Collins-Vergnioux \cite{BCV15} 
 by two quantum subgroups with amenable Haar traces. 
 Finally, in Section 4, using Theorem 3.3 and results of Brannan-Collins-Vergnioux \cite{BCV15} and 
 Brown-Dykema \cite{BD04}, we show that for $n \neq 3$, the Haar traces on 
 ${U_n^+}$ and ${O_n^+}$ are amenable, which means equivalently that 
 discrete quantum group duals $\widehat{U_n^+}$ and $\widehat{O_n^+}$ 
  have Kirchberg's factorization property.

\section{Preliminaries}
\label{prelim}

\iffalse
CQGs, Haar state/trace,  $U_n^+$, $O_n^+$

quantum groups generated by subgroups

Amenable traces, factorization property. 

free product of amenable traces
\fi

For convenience of the reader 
and to make this paper reasonably self-contained, 
we summarize here several main concepts and results that we will use. 

For simplicity of notation, $A \otimes B$ denotes the minimal tensor product  $A \otimes_{\rm min}  B$ 
of two $C^*$-algebras $A$ and $B$, which is the same as the spatial norm; 
$M_n = M_n(\CC)$ denotes the $n \times n$ complex matrix algebra. 
% and $\tr_n$ denotes its normalized trace.

There are many equivalent definitions of a compact quantum group (see Sect. 2 of \cite{free}). 
The tersest one is probably the following.

\begin{definition} 
{\rm (cf. \cite{Wor98a}) }
A {\bf compact quantum group} (CQG)  is a pair $G= (A, \Delta)$ where 
$A$ is a $C^*$algebra with unit and $\Delta: A \to A \otimes A$ 
is a morphism of unital $C^*$-algebras such that

(1) 
$( \Delta \otimes id)  \Delta  = (  id \otimes  \Delta)  \Delta $; 

(2) 
The vector spaces $(1 \otimes A)  \Delta(A)$
and  $(A \otimes 1)  \Delta(A)$ are dense in $A \otimes A$. 

\end{definition}

The $C^*$-algebra $A$ is also denoted by $C(G)$ and $A_G$. 
%We assume the $C^*$-norm on $C(G)$ to be universal. 
We use $\Delta_G$ to denote the coproduct $\Delta$ for $G$ for clarity when necessary.

One of the most important properties of a CQG is the existence of 
a {\bf Haar measure}, also called the {\bf Haar state}, as described in the theorem below. 
It gives rise to the Peter-Weyl theory of a CQG that ensures that the 
vector space ${\mathcal A}_G$ of matrix elements of all finite dimensional representations of $G$ is 
a Hopf $*$-algebra dense in $A$. The antipode $\kappa$ of ${\mathcal A}_G$ is an antihomomorphism 
and is given by $(id \otimes \kappa)(v) = v^{-1}$, where $v$ is an finite dimensional representation of $G$.  
%$(\kappa(v_{ij}))_{i,j=1}^d = (v_{ij})^{-1}$.  
%Denote the antipode of ${\mathcal A}_G$ is by $\kappa$.

Just as the group algebra of a discrete group, 
the $*$-algebra ${\mathcal A}_G$ has a universal (i.e. maximum) $C^*$-norm 
and a reduced (i.e. minimum) $C^*$-norm. 
 We assume the $C^*$-norm on $A=C(G)$ to be universal. 
The $C^*$-algebra $C(G)$ is also isomorphic to the universal group $C^*$-algebra 
$C^*(\hat{G})$ of the {\bf discrete quantum group dual} $\hat{G}$ of $G$. 
Our focus will be on the algebra $A=C(G)= C^*(\hat{G})$ without 
worrying about details of the duality theory or axiomatic theory of 
discrete quantum groups as found in \cite{PodWor90, EffRuan94,Daele96a}.

\begin{theorem} 
{\rm (Haar measure, cf. \cite{Wor87b, Wor98a})} \label{Haar}
Let $G= (A, \Delta)$ be a CQG.

{\rm (1)} 
 There is a unique state $h$ on $A$ such that for all $a$ in $A$, 
\begin{eqnarray} \label{Haarid}
(id \otimes h ) \Delta(a) = (h \otimes id) \Delta(a) = h(a)1_A. 
\end{eqnarray}
%$\varphi \star h = h \star \varphi = h$ for all states $\varphi$ on $A$. 
%$(\varphi \otimes h ) \Delta = (h \otimes \varphi) \Delta = h$ for all states $\varphi$ on $A$. 
{\rm (2)} 
The state $h$ is a trace
%, i.e., $h(ab)=h(ba)$ for all $a,b$ in $A$ 
if and only if %the antipode 
$\kappa$ has a bounded extension 
to a $*$-isomorphism from $A$ to its opposite $C^*$-algebra $A^{\rm op}$ such that 
$\kappa^2(a) = a$ for all $a$ in $ A$.
\end{theorem}

When (2) of the above holds, we say $G$ is a CQG of {\bf Kac type}, or 
$A$ is a compact type {\bf Kac algebra} \cite{EnSch}, and 
we continue to use  $\kappa$ to denote the extended anti-isomorphism. 
The Haar state will be also called the {\bf Haar trace} in this case.
It will be the main object of study in this paper. 

Quantum groups in this paper will be of Kac type unless otherwise stated. 

\begin{examples} (cf. \cite{free}) Let $n$ be a natural number. 

The {\bf universal unitary quantum group $U_n^+$} has quantum function algebra 
$A_u(n)=C(U_n^+)$ defined by generators and relations as follows: 
\begin{eqnarray}
C(U_n^+) = C^*\{ u_{ij} \; | \; 
% u u^* = I_n = u^* u, 
 u \bar{u}^t = I_n = \bar{u}^t u, u^t \bar{u} = I_n = \bar{u} u^t \}
\end{eqnarray}
where $u = (u_{ij})$, $u^t = (u_{ji})$, $\bar{u}= (u^*_{ij})$,  as elements in the 
$n \times n$ matrix algebra $M_n(C(U_n^+)) = M_n \otimes C(U_n^+)$ with entries in $C(U_n^+)$. 

The {\bf universal orthogonal quantum group $O_n^+$} has quantum function algebra 
$A_o(n)=C(O_n^+)$ defined by generators and relations as follows: 
\begin{eqnarray}
C(O_n^+) = C^*\{ u_{ij} \; | \;  u u^t = I_n = u^t u, \bar{u} = u \}
\end{eqnarray}
%where $u = (u_{ij})$, $u^t = (u_{ji})$ have similar meaning as above. 

The coproduct for both $C(U_n^+)$ and $C(O_n^+)$ are given by the same formula 

\begin{eqnarray}
\Delta(u_{ij}) = \sum_{k=1}^n u_{ik} \otimes u_{kj}.
\end{eqnarray}

Both $U_n^+$ and $O_n^+$ are CQGs of Kac type with the antipodes 
$\kappa$ defined by 

\begin{eqnarray}
\kappa(u_{ij}) = u_{ji}^*.  
\end{eqnarray}
\end{examples}

\begin{definition} 
A {\bf quantum subgroup} {\rm (cf. \cite[2.13]{free})}
of $ G $ is defined to be a pair $(K, \pi)$ where $K$ is a CQG and 
$\pi: C(G) \to C(K)$ is a surjection such that $(\pi \otimes \pi)\Delta_G = \Delta_{K} \pi $. 

According to Brannan, Collins and Vergnioux {\rm \cite[Definition 4]{BCV15}}, 
we say {\bf $G$ is generated by its quantum subgroups $(G_1, \pi_1)$ and $(G_2, \pi_2)$}, denoted 
$G =  \langle  G_1, G_2 \rangle $, if 
$$
{\rm Mor}_{G}(u, v) = {\rm Mor}_{G_1}(\pi_1(u), \pi_1(v)) \cap {\rm Mor}_{G_2}(\pi_2(u), \pi_2(v))
$$
for every pair of finite dimensional representations $u, v$ of $G$,  
where 
$
{\rm Mor}_{G}(u, v)
$
 is the linear space of intertwiners from $u$ to $v$ \cite{Wor87b,Wor88a}, and 
$\pi_i(u)$ denotes, by abuse of notation,  the representation $(id \otimes \pi_i)(u)$ 
of $G_i$ obtained as restriction of $u$ to $G_i$, $i=1,2$. 
This notion is also implicit in the proof of Theorem 3.1 in Chirvasitu {\rm \cite{Chir15}}. 

For continuous functionals  $\varphi, \psi$ on $C(G)$, $\varphi \star \psi$ denotes their {\bf convolution} 
defined by $\varphi \star \psi = (\varphi \otimes \psi) \Delta$, and 
$\varphi^{\star{ k}} = \varphi \star ... \star \varphi$ ($k$ times). 
\end{definition}

\begin{proposition}  \label{GeneratedQG}
{\rm (\cite[Proposition 3.5]{BCV15})}
Let $(G_1, \pi_1)$ and  $(G_2, \pi_2)$ be quantum subgroups of $G$. 
Let $h_G$ and $ h_{G_i}$ be  Haar states for $G$ and ${G_i}$ and let  
$\tau_i = h_{G_i} \pi_i$, $i=1,2$. Then the following are equivalent. 

{\rm (1)} $G =  \langle  G_1, G_2 \rangle $.
 
{\rm (2) } $\displaystyle h_G = \lim_{k \to \infty} (\tau_1 \star \tau_2)^{\star { k}}$ pointwise on $C(G)$. 

\end{proposition} 

The quantum group $O_n^{+}$ acts on ${\mathbb C}^n$ via 
$$\alpha(e_j) = \sum_{i=1}^n e_i \otimes u_{ij}, \; \; \; i, j = 1, ..., n$$ 
where $e_j$ ($j=1, ..., n$) is the standard basis for ${\mathbb C}^n$. 
Let $\xi \in {\mathbb C}^n$. Then the relation 
$\alpha (\xi) = \xi \otimes 1$ in terms of the basis $e_j$ 
defines a Woronowicz C$^*$-ideal (cf. \cite{simple}) $I_\xi$ of 
$C(O_n^{+})$.  
 Define $O_{n-1}^{+,\xi}$ to be the quantum subgroup of $O_n^{+}$ that fixes $\xi$, i.e., 
 $C(O_{n-1}^{+,\xi}) := C(O_n^{+})/I_\xi$. Note that $O_{n-1}^{+,\xi}$ is isomorphic to 
 $O_{n-1}^{+}$.  

\begin{theorem} \label{GeneratedO}
{\rm (\cite[Theorem 4.1]{BCV15})}
Let $n \geq 4$. The the following are true.

{\rm (1)} 
$O_n^+ =  \langle  O_n, O_{n-1}^{+, \; \xi} \rangle $ for each $\xi$ on the unit sphere $S^{n-1}$ in $\RR^n$, 
where 
%$O_{n-1}^{+, \; \xi}$ is the quantum subgroup of $O_n^+$ that fixes $\xi$, and 
$O_n$ is the classical orthogonal group. 

{\rm (2)} 
For any pair of linearly independent vectors  $\xi_1$,  $\xi_2$  on $S^{n-1}$, 
$O_n^+ =  \langle  O_{n-1}^{+, \; \xi_1}, O_{n-1}^{+, \; \xi_2} \rangle $.
\end{theorem}

Let $G_1$ and $G_2$ be CQGs. Then there is a CQG, denoted by $G_1 \hat{*} G_2$,  such 
that the free product $C(G_1) * C(G_2) = C(G_1 \hat{*} G_2)$ (cf.  \cite{free}).  
The notation $G_1 \hat{*} G_2$ is introduced in 2.1 of \cite{BCV15} referring to the reduced 
algebra $C_r (G_1) * C_r (G_2)$, but we prefer to use this to refer to the full 
$C^*$-algebra  $C(G_1) * C(G_2)$ under the universal norm.  
It makes no difference since the reduced and the full algebras 
describe the same quantum group (cf. III.7(2) in \cite{krein}).

\begin{theorem} \label{GeneratedevenO}
{\rm (\cite[Theorem 4.2]{BCV15})}
Let $S_{2n}$ be the permutation group on $2n$ symbols with $n \geq 2$. Then
$O_{2n}^+ =  \langle  S_{2n}, O_n^+ \hat{*} O_n^+ \rangle $.
\end{theorem}

Next we recall the notion of amenable traces on $C^*$-algebras and the 
Kirchberg's factorization property. 
The main theorem on this is due to Kirchberg (cf. \cite{Kirch94}), which is presented in more digestable 
form in \cite[Theorem 6.1]{Oza04},  and \cite[Theorem 6.2.7]{BO}. 
To state it, we fix some notation first. 

Let $A$ be a unital $C^*$-algebra and $\tau$ a trace on $A$ 
with the GNS-triple $(\pi_\tau, H_\tau, \Lambda_\tau(1))$, 
where 
$\Lambda_\tau(x)$ is the canonical image of $x \in A$ in the GNS Hilbert space 
$H_\tau$.
%$H_\tau:= L^2(A, \tau)$.
Denote by $\pi_\tau^{\rm op}$ the representation 
of the opposite $C^*$-algebra $A^{\rm op}$ of $A$ on $H_\tau$
defined by 
$$\pi_\tau^{\rm op} (b) \Lambda_\tau(x)= \Lambda_\tau(xb), \; \; \;  b, x\in A. $$ 

This gives rise to a representation $(\pi_\tau \cdot \pi_\tau^{\rm op})_{\rm alg}$
of $A\otimes_{\rm alg} A^{\rm op}$ on $H_\tau$ by 
$$
(\pi_\tau \cdot \pi_\tau^{\rm op})_{\rm alg} \colon A \otimes_{\rm alg} A^{\rm op}
 \ni\sum_k a_k \otimes {b_k}
 \longmapsto \sum_k\pi_\tau(a_k) \pi_\tau^{\rm op} (b_k)\in B(H_\tau)$$ 
which is continuous with respect to the maximal tensor norm.

The normalized trace on the $k \times k$  matrix algebra $M_k$ is denoted by $\tr_k$.

\begin{theorem}  \label{amentrace}
{\rm (\cite[Proposition 3.2]{Kirch94}, \cite[Theorem 6.1]{Oza04}, \cite[Theorem 6.2.7]{BO})} 

For a trace $\tau$ on a $C^*$-algebra $A$ in $B(H)$, 
the following are equivalent. 

{\rm (1)} 
The trace $\tau$ extends to an $A$-central state $\psi$ on $B(H)$,
i.e., the trace $\tau$ extends to a state $\psi$ on $B(H)$, such that 
$\psi(uxu^{*})=\psi(x)$ for every $x$ in $B(H)$ and  every unitary $u$ in $A$.

{\rm (2)}  \label{factnet}
There is a net of unital and completely positive (abbreviated UCP) maps $\varphi_n :  A \to M_{k_n}$  
such that $\displaystyle \tau(a) = \lim_n {\rm tr}_{k_n} (\varphi_n (a))$ 
and $\displaystyle \lim_n {\rm tr}_{k_n} (\varphi_n (b^*a) - \varphi_n (b^*) \varphi_n (a))=0$ 
for all $a,b$ in $A$. 

{\rm (3)} 
The representation $(\pi_\tau \cdot \pi_\tau^{\rm op})_{\rm alg}$ is continuous 
with respect to the minimal tensor norm on $A \otimes_{\rm alg} {A}^{\rm op}$. 
\end{theorem}

\begin{definition}
A trace $\tau$ is called {\bf amenable} if it satisfies the equivalent conditions in 
\thmref{amentrace}. The net $\varphi_n$ in 
%\thmref{amentrace}
(2) above is called a {\bf factorization net} for $\tau$, 
and the second condition in  
%\thmref{amentrace}
(2) is called the 
{\bf approximate multiplicativity condition}. 
\end{definition}

\begin{definition}
Let $G$ be a CQG of Kac type with Haar trace $h$. 
Its discrete quantum group dual $\hat{G}$ is said to have {\bf Kirchberg's factorization property} 
(or simply {\bf the factorization property}) if $\tau$ is amenable. 
\end{definition}

Denote $\pi_h$ by $\lambda$ and define representation $\rho$ of $C(G)$ 
by 
$$\rho(a)\Lambda_h(b) = \Lambda_h( b \kappa(a)), \; \; \; a, b \in C(G).$$ 
Then we can  replace the opposite $C^*$-algebra $C(G)^{\rm op}$ with the 
$C^*$-algebra $C(G)$  in the above formulation, and 
$\pi_h \otimes_{\rm alg} \pi_h^{\rm op}$ is replaced with 
$\lambda \otimes_{\rm alg} \rho$,  as stated in the introduction of this paper. 

\begin{remark} 
Note that if $\hat{G}$ has {\bf Kirchberg's factorization property}, then 
its Haar trace is hyperlinear, cf. \cite[Proposition 3.2]{Kirch94} and \cite[Definition 2]{BCV15}.
\end{remark}

\begin{proposition}  \label{amentrc}
{\rm (\cite[Proposition 6.3.7]{BO})}
The set of amenable traces on $A$ is a weak$^*$-closed convex face 
in the space of tracial states on $A$. 
\end{proposition}

\begin{proposition} \label{fpamentr}
{\rm (\cite[Proposition 7.3]{Oza04}, \cite[5.5]{BD04})}
If $\tau_i$ are amenable traces on $A_i$ ($i=1,2$), then the free product 
$\tau_1 \ast \tau_2$ is an amenable trace on $A_1 * A_2$. 
\end{proposition}

\section{Amenable Trace on Generated Quantum Subgroup}

In this section, we prove the general result that if $G_1$ and $G_2$ are 
CQGs with amenable Haar traces, then the quantum $ \langle  G_1, G_2 \rangle $ generated 
by them is a CQG with amenable Haar trace. This will be used later to prove the factorization 
property of discrete quantum groups $\widehat{U_n^+}$ and $\widehat{O_n^+}$.  

\begin{lemma} \label{amenquotient}
Let $(G_1, \pi_1)$ be a quantum subgroup of a CQG $G$ with surjection 
$\pi_1: C(G) \rightarrow C(G_1)$, and amenable Haar trace $ h_{G_1}$. 
Let $\tau_1 = h_{G_1} \circ  \pi_1$.
%, where $h_1$ is the state on $C(G)$
Then $\tau_1$ is an amenable trace on $C(G)$. 
\end{lemma}

\begin{proof}
This is essentially Proposition 6.3.5(2)  in \cite{BO}. 
We include a complete proof however as \cite{BO} does not provide an 
explicit proof but leaves it for the reader. 

Let $j: C(G) \hookrightarrow B(H_{G})$ and 
$j_1: C(G_1) \hookrightarrow B(H_{G_1})$  be faithful non-degenerate representations.

Extend $j_1 \pi_1: C(G) \rightarrow B(H_{G_1})$ to UCP maps 
$\Phi_1: B(H_{G})  \rightarrow B(H_{G_1})$ 
using Arveson extension theorem. 
Now consider the amenable extension of $h_{G_1}$ to \\
$\tilde{h}_{G_1}: B(H_{G_1})  \rightarrow {\mathbb C}$. 

Consider the functionals $\tilde{h}_1= \tilde{h}_{G_1} \circ  \Phi_1: B(H_{G})  \rightarrow {\mathbb C}$. 
Evidently $\tilde{h}_1$ is a state on $B(H_G)$. 
%For any $u \in U(j(C(G)))$, the group of unitary elements in $j(C(G))$, and $T \in B(H_G)$, 
For any unitary element $u$ in $j(C(G))$, and $T \in B(H_G)$, 
we compute
\begin{eqnarray*}
\tilde{h}_1 (u^*Tu) 
&=& \tilde{h}_{G_1} \circ  \Phi_1 (u^*Tu) \\ 
&=& \tilde{h}_{G_1} (\Phi_1 (u)^* \Phi_1(T) \Phi_1(u)) \\
&=& \tilde{h}_{G_1} (j_1 \pi_1 (u)^* \Phi_1(T) j_1 \pi_1(u)),
\end{eqnarray*}
where in the second identity above, we use the fact that 
$u$ is in the multiplicative domain of $\Phi_1$ because   $\Phi_1$ is the 
Arveson extension of $j_1 \pi_1$. 
Since $C(G_1)$ is non-degneratedly represented in $B(H_{G_1})$, 
$j_1 \pi_1(u)$ is a unitary element in the image $j_1(C(G_1))$. 
Now the assumption that $\tilde{h}_{G_1}$  is amenable extension of $h_{G_1}$ 
gives 
\begin{eqnarray*}
\tilde{h}_{G_1} (j_1 \pi_1 (u)^* \Phi_1(T) j_1 \pi_1(u)) 
= \tilde{h}_{G_1} ( \Phi_1(T) )= \tilde{h}_1 (T).
\end{eqnarray*}
In addition, for $a \in C(G)$, 
\begin{eqnarray*}
\tilde{h}_1 (a) &=& \tilde{h}_{G_1} ( \Phi_1(a) )  \\
                  &=& \tilde{h}_{G_1} ( \pi_1(a) ) \; \; \; \; (\mbox{$\Phi_1$ is an extension of $\pi_1$})  \\
                  &=&  {h}_{G_1} ( \pi_1(a) ) = \tau_1 (a).
\end{eqnarray*}

This shows that $\tau_1$ is an amenable trace on $C(G)$, since the amenability of trace is independent of any 
faithful representation.
\end{proof}

\begin{lemma}   \label{amenconvol}
Let $\tau_1$ and $\tau_2$ be amenable traces on $C(G)$. Then their convolution  
$\tau_1 \star \tau_2 := (\tau_1 \otimes \tau_2) \Delta_G $ is also an amenable trace. 
\end{lemma}
\begin{proof}
%\pf
Given that $\tau_1$ is an amenable trace on $C(G)$, by \thmref{amentrace}, there exists a net
of UCP maps $\{ \varphi_n^1 \}$ such that 
\begin{eqnarray}
& & \varphi_n^1: C(G) \rightarrow M_{k_n} \\
& & \tau_1 = \lim_n {\rm tr}_{k_n} \circ \varphi_n^1  \; \; {\mbox{pointwise}}  \\
& & \lim_n {\rm tr}_{k_n} ( \varphi_n^1(b^* a) - \varphi_n^1 (b)^* \varphi_n^1 (a) )=0, 
\; \; {\mbox{for all $a, b \in C(G)$}}. \label{approxmulass}
\end{eqnarray}
Similarly for amenable trace $\tau_2$ on $C(G)$, there exists a net
of UCP maps $\{ \varphi_m^2 \}$ with the same properties above. 
 
Now the state $\tau_1 \star \tau_2$ on $C(G)$ is defined by 
$$\tau_1 \star \tau_2 = (\tau_1 \otimes \tau_2) \Delta_G,$$
where $\Delta_G$ is the coproduct on $C(G)$. It is trivial to check that 
if $\tau_1$ and $\tau_2$ are traces on $C(G)$, then so is
$\tau_1 \star \tau_2$. 

Consider the net of 
UCP maps $\varphi_{(n,m)}$ defined by 
\begin{eqnarray*}
\varphi_{(n,m)} = (\varphi_n^1 \otimes \varphi_m^2) \Delta_G: 
C(G) \rightarrow M_{k_n} \otimes  M_{k_m} \cong M_{k_nk_m}. 
\end{eqnarray*}
We claim that the $\varphi_{(n,m)}$'s are a factorization net for $\tau_1 \star \tau_2$. 
We proceed by checking the two conditions in \thmref{amentrace}(2).

For $a \in {\mathcal A}$, 
let $\displaystyle \Delta_G(a) = \sum_{i} a_{1i} \otimes a_{2i}$, which is a finite sum by property of 
the Hopf algebra ${\mathcal A}$. Then 
$$
\varphi_{(n,m)}(a) = \sum_i \varphi_n^1 (a_{1i}) \otimes \varphi_m^2 (a_{2i})
$$
$$
({\rm tr}_{k_n} \otimes {\rm tr}_{k_m}) \varphi_{(n,m)}(a)
 = \sum {\rm tr}_{k_n} \varphi_n^1 (a_{1i}) {\rm tr}_{k_m} \varphi_m^2 (a_{2i}).
$$
Since $ \varphi_n^1$ (resp. $\varphi_m^2$) is a factorization net for $\tau_1$ (resp. $\tau_2$), we have 
\begin{eqnarray*}
& & \lim_{(n,m)} ({\rm tr}_{k_n} \otimes {\rm tr}_{k_m}) \varphi_{(n,m)}(a) 
= \lim_{(n,m)} \sum {\rm tr}_{k_n} \varphi_n^1 (a_{1i}) {\rm tr}_{k_m} \varphi_m^2 (a_{2i}) \\
&=&   \sum \tau_1 (a_{1i}) \tau_2 (a_{2i}) = (\tau_1 \otimes \tau_2) \Delta_G(a) \\
&=& \tau_1 \star \tau_2 (a)
\end{eqnarray*}
Since ${\mathcal A}$ is dense in $C(G)$, the above holds for all $a \in C(G)$.

Next, we check the approximate multiplicativity condition of $\varphi_{(n,m)}$ 
in \thmref{amentrace}(2), i.e.,  
\begin{eqnarray} \label{multab}
({\rm tr}_{k_n} \otimes {\rm tr}_{k_m}) (\varphi_{(n,m)}(b^*a) - \varphi_{(n,m)}(b^*)  \varphi_{(n,m)}(a) ) \rightarrow 0
\end{eqnarray}
for $a, b \in C(G)$. Note that it is in fact enough to check only 
\begin{eqnarray} \label{multb}
({\rm tr}_{k_n} \otimes {\rm tr}_{k_m}) (\varphi_{(n,m)}(b^*b) - \varphi_{(n,m)}(b^*)  \varphi_{(n,m)}(b) ) \rightarrow 0
\end{eqnarray}
for $b \in C(G)$. 

To see this,  any UCP map $\theta: A \rightarrow M_k$ from a $C^*$-algebra $A$  
%${\rm tr} \theta$ 
gives rise to a positive semi-definite form on $A$ defined by 
\begin{eqnarray}
 \langle  x, y \rangle  := {\rm tr} (\theta(y^*x) - \theta(y^*)\theta(x)), \; \; x, y \in A.
\end{eqnarray}

Hence this satisfies the Cauchy-Schwarz inequality 
\begin{eqnarray}
| \langle  x,y \rangle | \leq \| x \| \| y \|, \; \;  \| x \| :=  \langle  x , x  \rangle ^{\frac{1}{2}}.
\end{eqnarray}

%This is the same as  \cite{Oza04}, last line on p522. 

Applying this with $x=a, \; y=b$ for the corresponding UCP's above, we see (\ref{multab}) follows 
from (\ref{multb}). 

To check (\ref{multb}), 
let $b \in {\mathcal A}$ and  
 $\displaystyle \Delta_G(b) = \sum_{i} b_{1i} \otimes b_{2i}$ (a finite sum). 
Then
\begin{eqnarray*}
\Delta(b^*b) &=& \Delta(b^*) \Delta(b) \\
%&=& (\sum_{i} b^*_{1i} \otimes b^*_{2i}) (\sum_{i} b_{1j} \otimes b_{2j}) \\
&=& \sum_{i, j} b^*_{1i} b_{1j} \otimes  b^*_{2i} b_{2j}.
\end{eqnarray*}
Hence, 
$$\varphi_{(n,m)}(b^*b) = \sum_{i, j} \varphi_n^1(b^*_{1i} b_{1j}) \otimes  \varphi_m^2(b^*_{2i} b_{2j}).$$  
 Also,  
\begin{eqnarray*}
\varphi_{(n,m)}(b^*) \varphi_{(n,m)}(b) &=& \sum_{i} 
\varphi_n^1(b_{1i})^* \otimes  \varphi_m^2(b_{2i})^* \sum_{j} \varphi_n^1(b_{1j}) \otimes  \varphi_m^2(b_{2j}) \\
&=&  \sum_{i, j} \varphi_n^1(b_{1i})^* \varphi_n^1(b_{1j}) \otimes  \varphi_m^2(b_{2i})^* \varphi_m^2(b_{2j})  
\end{eqnarray*}
Combining the above,  
\begin{eqnarray*}
& & ({\rm tr}_{k_n} \otimes {\rm tr}_{k_m}) (\varphi_{(n,m)}(b^*b) - \varphi_{(n,m)}(b)^*  \varphi_{(n,m)}(b) ) \\
&=&   ({\rm tr}_{k_n} \otimes {\rm tr}_{k_m}) \sum_{i, j} \varphi_n^1(b^*_{1i} b_{1j}) \otimes  \varphi_m^2(b^*_{2i} b_{2j}) \\
& & - ({\rm tr}_{k_n} \otimes {\rm tr}_{k_m}) 
\sum_{i, j} \varphi_n^1(b_{1i})^* \varphi_n^1(b_{1j}) \otimes  \varphi_m^2(b_{2i})^* \varphi_m^2(b_{2j})  \\
&=&   \sum_{i, j} {\rm tr}_{k_n} ( \varphi_n^1(b^*_{1i} b_{1j}) )  {\rm tr}_{k_m} ( \varphi_m^2(b^*_{2i} b_{2j}))  \\ 
& & -  \sum_{i, j} {\rm tr}_{k_n} ( \varphi_n^1(b_{1i})^* \varphi_n^1(b_{1j})  {\rm tr}_{k_m} 
( \varphi_m^2(b_{2i})^* \varphi_m^2(b_{2j}) )  \; \; \; \; \; \;  
(*) 
%\label{(*)}
\end{eqnarray*}

\iffalse
By assumption in (\ref{approxmulass}), 
$$
\lim_n {\rm tr}_{k_n} ( \varphi_n^1(b^*_{1i} b_{1j})  - \varphi_n^1(b_{1i})^* \varphi_n^1(b_{1j}))  = 0,  
$$
$$
\lim_m {\rm tr}_{k_m} ( \varphi_m^2(b^*_{2i} b_{2j})  - \varphi_m^2(b_{2i})^* \varphi_m^2(b_{2j}))  = 0. 
$$

Applying this to the last expression 
\fi

Applying the assumption in (\ref{approxmulass}) for $\varphi_n^1$ and a similar one  for $\varphi_m^2$ 
 to the last expression  ($*$) 
for the computation of % \\
$({\rm tr}_{k_n} \otimes {\rm tr}_{k_m}) (\varphi_{(n,m)}(b^*b) - \varphi_{(n,m)}(b)^*  \varphi_{(n,m)}(b) )$ 
above, we obtain 
\begin{eqnarray*}
& & \lim_{(n,m)}   \sum_{i, j} %\left[ 
[{\rm tr}_{k_n} ( \varphi_n^1(b^*_{1i} b_{1j}) )  {\rm tr}_{k_m} ( \varphi_m^2(b^*_{2i} b_{2j}))  \\
& & -  \sum_{i, j} {\rm tr}_{k_n} 
( \varphi_n^1(b_{1i})^* \varphi_n^1(b_{1j})  {\rm tr}_{k_m} ( \varphi_m^2(b_{2i})^* \varphi_m^2(b_{2j}) )]  %\right] 
=0.
\end{eqnarray*}
Therefore we have shown that for $b \in {\mathcal A}_{G}$, 
$$\lim_{(n,m)} ({\rm tr}_{k_n} \otimes {\rm tr}_{k_m}) 
(\varphi_{(n,m)}(b^*b) - \varphi_{(n,m)}(b)^*  \varphi_{(n,m)}(b) ) = 0.$$
This continues to hold for any $b \in C(G)$ by the density of ${\mathcal A}_{G}$ in $C(G)$. 
%\qed
\end{proof}

\begin{theorem} \label{amenGeneratedQG}
Let $G$ be a CQG with tracial Haar state $h_G$. 
Let $G_1$ and $G_2$ be quantum subgroups of $G$ with tracial Haar states $h_{G_1}$ and $h_{G_2}$. 
Further assume that $h_{G_1}$ and $h_{G_2}$ are amenable traces. Then 
$h_G$ is amenable if $G =  \langle  G_1, G_2 \rangle $.
\end{theorem}

\begin{proof}
From \propref{GeneratedQG}(3), we have
$$
h_G = \lim_{k \rightarrow \infty} (\tau_1 \star \tau_2)^{\star k}
$$
pointwise on ${\mathcal A}_G$,
where each $\tau_i = h_{G_i}\pi_i$ ($i=1, \; 2$) is amenable by  
\lmref{amenquotient}. 
By 
 %Lemma 
 \lmref{amenconvol}
all traces of the form $(\tau_1 \star \tau_2)^{\star k}$  on $C(G)$ are amenable. 
Therefore, on the dense subalgebra ${\mathcal A}_G$, $h_G$ is approximated by 
amenable traces pointwise. Usual density arguments extend it to the full algebra 
$C(G)$. Hence $h_G$ is a weak*-limit of amenable traces and is therefore an amenable 
traces by \propref{amentrc}. %Proposition 6.3.7 in \cite{BO}. 
\end{proof}

\section{Kirchberg's factorization property for \\ 
discrete quantum groups $\widehat{U_n^+}$ and $\widehat{O_n^+}$, $n \neq 3$}

In this section, we show that  for $n \neq 3$, the discrete quantum group duals $\widehat{U_n^+}$ and $\widehat{O_n^+}$ 
have Kirchberg's factorization property (we will simply call this the factorization property). 

Recall that a for a compact quantum group $G$
%$G = (A_G, \Delta_G)$ 
with tracial Haar 
measure $h_G$, the discrete dual $\hat{G}$ has the factorization property if and 
only if $h_G$ is amenable. We will show that  $h_G$ is amenable when $G$ is one of 
${U_n^+}$ and ${O_n^+}$ and $n \neq 3$.

\begin{lemma} \label{amenO_2^+}
$\widehat{O_2^+}$ has the factorization property.
\end{lemma}

\begin{proof}
Since $C(O_2^+)$ is a nuclear $C^*$-algebra, every trace on  it is amenable. 
\end{proof}

\begin{lemma} \label{amenO_2^+fp}
$\widehat{O_2^+} \, {*} \, \widehat{O_2^+}$ has the factorization property.
\end{lemma}

\begin{proof}
The statement of the lemma means 
the Haar state  on $C(O_2^+) {*} C(O_2^+)$ is an amenable trace on account 
of $C(O_2^+) {*} C(O_2^+) = C^*(\widehat{O_2^+} \, {*} \, \widehat{O_2^+})$.

By \cite{free}, the Haar state  on $C(O_2^+) {*} C(O_2^+)$ is the free product 
 $h_{O_2^+} * h_{O_2^+}$, which is a trace by \cite{Avitz82} since $h_{O_2^+}$ is one. 
 By \propref{fpamentr} along with the above lemma,  $h_{O_2^+} * h_{O_2^+}$ is amenable.
\end{proof}

\begin{theorem} \label{factpO}
$\widehat{O_n^+}$ has the factorization property for $n \neq 3$.
\end{theorem}

\begin{proof}
The case  $n=2$ is \lmref{amenO_2^+}.  

For $n=4$, by \thmref{GeneratedevenO}, we have 
$$
%O_4^+ =   \langle  {\mathfrak S}_4, O_2^+ \hat{*} O_2^+ \rangle 
O_4^+ =   \langle  {S}_4, O_2^+ \hat{*} O_2^+ \rangle 
$$
where ${S}_4$ is the ordinary permutation group. 
Since the commutative algebra $C(S_4)$ is nuclear (so its Haar state  is amenable), by combining 
%Lemma 
\lmref{amenO_2^+fp} 
and \thmref{amenGeneratedQG}, 
we deduce that 
$\widehat{O_4^+}$ has the factorization property. 

For $n>4$, by \thmref{GeneratedO}
%\cite{BCV15}   
$$
O_n^+ =  \langle  O_n, O_{n-1}^{+, \, \xi}  \rangle , \; \; \xi \in S^{n-1},
$$
where $ S^{n-1}$ is the unit sphere in ${\mathbb R}^n$ centered at the origin.
As $O_{n-1}^{+, \, \xi} \cong O_{n-1}^{+}$ for all $\xi$ and $O_n$ is a 
ordinary compact group, by using 
%Theorem  
\thmref{amenGeneratedQG}  
%Theorem 2.3 
again and induction we conclude that $\widehat{O_n^+}$ has the factorization property.
\end{proof}

\begin{theorem}  \label{factpU}
$\widehat{U_n^+}$ has the factorization property for $n \neq 3$.
\end{theorem}

\begin{proof}
In \cite{Banica97a}, Banica showed that the morphism of Woronowicz $C^*$-algebras 
\begin{eqnarray}
\pi: C({U_n^+}) \rightarrow C({\mathbb T}) * C({O_n^+}), \; \;
\pi(u_{ij}) = za_{ij}
\end{eqnarray}
 restricts to an injection of Hopf $*$-algebras 
\begin{eqnarray}
\pi: 
{\mathcal A}_{U_n^+} \hookrightarrow {\mathcal A}_{\mathbb T} * {\mathcal A}_{O_n^+},   
\end{eqnarray}
where we use $a_{ij}$ to denote the generators of $C({O_n^+})$. 
By property of the Haar state (see formula (\ref{Haarid})), 
the Haar state  $h_{U_n^+}$ on ${\mathcal A}_{U_n^+}$ 
is the restriction of the Haar state  $h_{\mathbb T} * h_{O_n^+}$
on ${\mathcal A}_{{\mathbb T}} * {\mathcal A}_{O_n^+}$:  
\begin{eqnarray} \label{decomposeHaar}
h_{U_n^+} = (h_{\mathbb T} * h_{O_n^+}) \circ \pi.
\end{eqnarray}
By density of ${\mathcal A}_{U_n^+}$ in $C({U_n^+})$, (\ref{decomposeHaar}) is still valid on 
$C({U_n^+})$. 
By \propref{fpamentr} and  \thmref{factpO}, 
$h_{\mathbb T} * h_{O_n^+}$ is amenable. Hence there exists a net of UCP maps 
$\varphi_n : C({\mathbb T}) * C(O_n^+) \rightarrow M_{k_n}$ with factorization property. 
Consider the maps 
\begin{eqnarray} \label{factmap4U}
\varphi_n \circ \pi: 
C({U_n^+}) \rightarrow M_{k_n}.  
\end{eqnarray}

It is easy to see that, 
$$ \lim_{n \rightarrow \infty} {\rm tr}_{k_n} (\varphi_n \pi (ab) - \varphi_n \pi(a) \varphi_n \pi(b) )  = 0, $$
%$$ \lim_{n \rightarrow \infty} \|\varphi_n \pi (ab) - \varphi_n \pi(a) \varphi_n \pi(b) \|_{2, {\rm tr}_{k_n}} = 0, $$
for $a, b \in C({U_n^+})$, where ${\rm tr}_{k_n}$ is the normalized trace on $M_{k_n}$. 

We also have that 
$$
\lim_{n \rightarrow \infty} {\rm tr}_{k_n} \varphi_n \pi (a) = (h_{\mathbb T} * h_{O_n^+})(\pi(a)) = h_{U_n^+}(a)
$$
for $a \in {\mathcal A}_{U_n^+}$.  
The last limit obviously extends to all 
$a \in C({U_n^+})$ by the usual density argument. 

This proves that $h_{U_n^+}$ is amenable.
\end{proof}

\begin{remark}
By virtue of Theorem 6.2.7 in \cite{BO},  the exact same argument after  (\ref{decomposeHaar}) %(\ref{factmap4U}) 
above proves Proposition 6.3.5(2)  in \cite{BO} again, as the proof of 
%Lemma 
\lmref{amenquotient} does. 
Namely, if $J \triangleleft A$ is a closed ideal and $\tau$ is an amenable trace on $A/J$, then the 
induced trace $\tau \pi$ on $A$ is also amenable, where $\pi: A \rightarrow A/J$. 
\end{remark}

\begin{problem}
Just as the main results in \cite{Chir15} and \cite{BCV15}, 
Theorems \ref{factpO} and \ref{factpU} excludes the case $n=3$.  
However, we believe that all these results should equally hold for $n=3$. 
To prove that, we may need to develop new tools or a different method of proof. 
\end{problem}

\begin{problem}
For other quantum groups of Kac type such as the quantum automorphism groups in 
\cite{qsym,Banica99a}, it would be interesting to determine if they have the factorization property, 
though we have the following obvious result because $C(S^+_4)$ is nuclear \cite{BanMor07}, 
$C(S^+_4)$ being the notation of Banica {\sl et al.} for $A_{aut}(X_4)$ in \cite{qsym}. 
\end{problem}

\begin{proposition}
The discrete dual $\widehat{S^+_4}$ of the quantum permutation group ${S}^+_4$  
has the factorization property.
\end{proposition}

In forthcoming work, we will generalize Kirchberg's factorization property to 
non-Kac type discrete quantum groups.

\bigskip
\noindent
{\bf Acknowledgments.}
The authors thank the referees for careful reading of the manuscript.

\iffalse

\newpage

\pagenumbering{arabic}
\setcounter{page}{1}

\fi

%\newpage

\end{document}